\documentclass[draft]{article}
\usepackage{amsfonts}
\usepackage{amssymb}
\usepackage{amsmath}
\usepackage{amsthm}

\newtheorem{theorem}{Theorem}
\newtheorem{proposition}{Proposition}

\newtheorem{lemma}{Lemma}
\newtheorem{corollary}{Corollary}

\newcommand{\bbC}{{\mathord{\mathbb{C}}}}
\newcommand{\bbD}{{\mathord{\mathbb{D}}}}

\newcommand{\bbR}{{\mathord{\mathbb{R}}}}
\newcommand{\bbT}{{\mathord{\mathbb{T}}}}
\newcommand{\bbZ}{{\mathord{\mathbb{Z}}}}

\newcommand{\re}{\mathop{\mathrm{Re}}}

\newcommand{\supp}{\mathop{\mathrm{supp}}}

\newcommand{\Diff}{\mathop{\mathrm{Diff}}\nolimits}

\newcommand{\GL}{\mathop{\mathrm{GL}}}

\newcommand{\Hom}{\mathop{\mathrm{Hom}}\nolimits}

\newcommand{\Sp}{\mathop{\mathrm{Sp}}\nolimits}
\newcommand{\SU}{\mathop{\mathrm{SU}}}

\newcommand{\sll}{\mathop{\mathrm{sl}}}

\newcommand{\Ad}{\mathop{\mathrm{Ad}}}

\newcommand{\fin}{{\mathord{\mathrm{fin}}}}
\newcommand{\cA}{{\mathord{\mathcal{A}}}}

\newcommand{\cM}{{\mathord{\mathcal{M}}}}

\newcommand{\cP}{\mathord{\mathcal{P}}}

\newcommand{\cN}{\mathord{\mathcal{N}}}
\newcommand{\cH}{\mathord{\mathcal{H}}}

\newcommand{\cR}{\mathord{\mathcal{R}}}

\newcommand{\frA}{\mathord{\mathfrak{A}}}

\newcommand{\frg}{\mathord{\mathfrak{g}}}

\newcommand{\frt}{\mathord{\mathfrak{t}}}

\newcommand{\al}{{\mathord{\alpha}}}

\newcommand{\vf}{{\mathord{\varphi}}}
\newcommand{\si}{{\mathord{\sigma}}}

\newcommand{\vk}{{\mathord{\varkappa}}}
\newcommand{\la}{{\mathord{\lambda}}}

\newcommand{\eps}{{\mathord{\epsilon}}}
\newcommand{\ze}{{\mathord{\zeta}}}
\newcommand{\one}{{\mathord{\mathbf1}}}

\newdimen\theight
\newcommand\mrk[1]{%
             \vadjust{\setbox0=\hbox{\kern10pt\quad#1}%
             \theight=\ht0
             \advance\theight by \dp0    \advance\theight by \lineskip
             \kern -\theight \vbox to \theight{\rightline{\rlap{\box0}}%
             \vss}%
             }}%

\let\wh=\widehat
\let\ov=\overline
\let\td=\tilde

\title{\bf Invariant function algebras on compact commutative homogeneous spaces
}
\author{V.M. Gichev\thanks{Partially supported by RFBR grants
06-08-01403, 06-07-89051.}}
\date{}
\begin{document}
\maketitle
\begin{center}
{\it Dedicated to E.B.\,Vinberg on the occasion of his 70th
birthday}
\end{center}

{\abstract{Let $M$ be a commutative homogeneous space of a compact
Lie group $G$ and $A$ be a closed $G$-invariant subalgebra of the
Banach algebra $C(M)$. A function algebra is called antisymmetric
if it does not contain nonconstant real functions. By the main
result of this paper, $A$ is antisymmetric if and only if the
invariant probability measure on $M$ is multiplicative on $A$.
This implies, for example, the following theorem: if $G^\bbC$ acts
transitively on a Stein manifold $\cM$, $v\in\cM$, and the compact
orbit $M=Gv$ is a commutative homogeneous space, then $M$ is a
real form of $\cM$.}}

\section{Introduction}
Let $M=G/H$ be a homogeneous space of a compact connected Lie
group $G$, where $H$ is the stable subgroup of the base point $o$.
$M$ is called {\it multiplicity free} if the quasiregular
representation of $G$ in $C(M)$ contains every irreducible
representation of $G$ with a multiplicity $0$ or $1$. This
remarkable class of homogeneous spaces can be characterized by
each of the following properties (see \cite{Vi} or \cite{W07}; the
noncompact case is more complicated):
\begin{itemize}
\item[\rm 1)] $M$ is {\it commutative}, i.e., the algebra of all
invariant differential operators on $M$ is commutative; \item[\rm
2)] $(G,H)$ is a {\it Gelfand pair}, i.e., the convolution algebra
of all left and right $H$-invariant functions in $L^1(G)$ (or
measures) is commutative; \item[\rm 3)] a generic $G$-orbit in the
cotangent bundle $T^*M$ is coisotropic (a submanifold $L$ of a
symplectic manifold is called {\it coisotropic} if
$T_pL^\bot\subseteq T_pL$
for each $p\in L$); 
\item[\rm 4)] $M$ is {\it weakly commutative} (the Poisson algebra
$C^\infty(T^*M)$ is commutative); \item[\rm 5)] $M$ is {\it weakly
symmetric} (there exists an automorphism $\vk$ of $G$ such that
$\vk(g)\in Hg^{-1}H$ for all $g\in G$); \item[\rm 6)]
$M^\bbC=G^\bbC/H^\bbC$ is {\it spherical}, i.e., a Borel subgroup
of $G^\bbC$ has an open orbit in $M^\bbC$.
\end{itemize}
This class has been intensively investigated last years
(\cite{AV}, \cite{BeRa}-\cite{Br86}, \cite{L96}-\cite{L01},
\cite{Pa94}-\cite{Pa03}, \cite{Vi}, \cite{W07}; the list is far
from being complete). An important example of a commutative
homogeneous space is the group $G$ acting on itself by left and
right translations.

The space $C(M)$ of all continuous complex valued functions on $M$
equipped with the sup-norm is a Banach algebra. We say that $A$ is
an {\it invariant function algebra on $M$} if $A$ is a
$G$-invariant closed subalgebra of $C(M)$ which contains constant
functions. Invariant function algebras have been studied since 50s
as a natural generalization of algebras of analytic functions
(\cite{Ag1}-\cite{AgV}, \cite{AS}-\cite{Bj}, \cite{Ga}, \cite{Gl},
\cite{Ri}-\cite{RuBS}, \cite{Wo}). The structure of bi-invariant
function algebras on compact groups is well understood. In 1965,
Gangolli (\cite{Ga}) and Wolf (\cite{Wo}) proved that each
bi-invariant algebra on a semisimple group is self-conjugate with
respect to the complex conjugation. Then, due to the
Stone--Weierstrass theorem, the problem becomes purely geometric.
On the other hand, there are many bi-invariant algebras on tori
$\bbT^n$: they are in one-to-one correspondence with semigroups in
the dual group $\bbZ^n$ and may be {\it antisymmetric} (i.e.,
contain no nonconstant real function). Algebras of analytic
functions on complex analytic varieties are antisymmetric. The
converse is not true: there are examples of antisymmetric function
algebras on compact subsets of $\bbC^n$ with no analytic structure
in their maximal ideal spaces (see, for example, \cite{St}).
However, antisymmetric invariant function algebras have nontrivial
analytic structure in all known examples.

The maximal ideal space $\cM_A$ of a bi-invariant algebra $A$ has
a natural semigroup structure (with $G$ as the group of units) and
a compatible analytic structure (\cite{AS}, \cite{Gi79}).
Idempotents of the semigroup correspond to the Haar measures of
some subgroups. The following criterion of antisymmetry holds: a
bi-invariant algebra $A$ on a compact group $G$ is antisymmetric
if and only if the Haar measure of $G$ is multiplicative on $A$.
For generic homogeneous spaces, this is not true.

In this paper, we prove  that the criterion above holds for
commutative homogeneous spaces: an invariant function algebra on
such a space is antisymmetric if and only if the invariant
probability measure is multiplicative on the algebra
(Theorem~\ref{sphan}). The part ``if'' is easy. To prove the
converse, we use a result of Latypov (\cite{La99}): all invariant
function algebras on a homogeneous space $M=G/H$ are
self-conjugate with respect to the complex conjugation if and only
if $M$ satisfies the following property (F): the group $N/H$,
where $N$ is the normalizer of $H$ in $G$, is finite.
For $M^\bbC$, this means that it is closed in any ambient affine
$G^\bbC$-manifold.

If an invariant algebra $A$ on a commutative homogeneous space $M$
is antisymmetric and finitely generated (i.e., generated by a
finite dimensional invariant subspace), then $\cM_A$ can be
realized as the polynomial hull of a $G$-orbit in a nilpotent
spherical $G^\bbC$-orbit in a $G^\bbC$-module
(Theorem~\ref{ckanm}, (b)). Such orbits were studied by Panyushev
in papers \cite{Pa94}-\cite{Pa03}. They have many remarkable
properties and admit a simple description for some special
$G^\bbC$-modules.

Every invariant function algebra $A$ on a commutative homogeneous
space $M$ is completely determined by its ``antisymmetric part''.
$\cM_A$ is fibred over a commutative homogeneous space $\td M$; a
fibre is the maximal ideal space of an antisymmetric invariant
function algebra on a commutative homogeneous space. There is a
natural surjective homomorphism $A\to C(\td M)$, whose kernel is
the maximal proper invariant ideal of $A$ (Theorem~\ref{geifa},
Theorem~\ref{hulls}). In particular, $A$ is self-conjugate if and
only if it contains no proper invariant ideal.

The following theorem is a consequence of the mentioned results.
\begin{theorem}\label{comac}
Let $G$ be a compact Lie group and $\cM$ be a complex analytic
Stein manifold. Suppose that $G^\bbC$ acts transitively on $\cM$
by holomorphic transformations. Let $v\in\cM$. If $M=Gv$ is a
commutative homogeneous space, then $M$ is a real form of $\cM$.
\end{theorem}
``A real form'' above means that there exists an antiholomorphic
involution of $\cM$ which commutes with $G$ such that $M$ is the
set of its fixed points. Then $\cM\cong M^\bbC$, where
$M^\bbC=G^\bbC/H^\bbC$ and $H$ is the stable subgroup of $v$ in
$G$.
\begin{corollary}\label{nooth}
Let $M=G/H$ be a commutative homogeneous space. If $M$ satisfies
{\rm (F)}, then $M^\bbC=G^\bbC/H^\bbC$ contains no commutative
$G$-orbit, except for $M$.
\end{corollary}
In what follows and in the theorem above, $G$ may be
non-connected; we say that $M=G/H$ is {\it commutative} if it is
multiplicity free. This is true if and only if $(G,H)$ is a
Gelfand pair (see, for example, \cite[Theorem~9.7.1]{W07}). We
need no other characterization but keep the term ``commutative''
due to the following equivalent property: the algebra of all
bounded operators in $L^2(M)$ which commute with $G$ is
commutative.

The criterion of antisymmetry was proved in preprint \cite{GiP},
which also contains an infinite dimensional version of the
Hilbert--Mumford criterion for commutative orbits.

\section{Auxiliary material}
The dual space of a Banach space $X$ is denoted by $X'$. The {\it
maximal ideal space} $\cM_A$ ({\it spectrum}) of a commutative
Banach algebra $A$ with identity element $\one$ is the set
$\mathop{\rm Hom}(A,\bbC)\subset A'$ of all nonzero multiplicative
linear functionals equipped with the ${}*{}$-weak topology of the
dual space $A'$. A {\it function algebra} on a compact Hausdorff
topological space $Q$ is a closed subalgebra of the Banach algebra
$C(Q)$, endowed with the sup-norm,  which contains all constant
functions; the set of them we identify with $\bbC$. Let $A$ be a
function algebra on $Q$.  A set $Y\subseteq Q$ is called {\it a
set of antisymmetry for $A$} if every function which takes only
real values on $Y$ is constant on $Y$. If two sets of antisymmetry
overlaps, then their union is also a set of antisymmetry. Hence,
any such a set is contained in a maximal one, which is necessarily
closed. Distinct maximal sets of antisymmetry are disjoint. This
defines a foliation in $Q$. A function algebra on $Q$ is called
{\it antisymmetric} if $Q$ is a set of antisymmetry. We denote by
$\frA(A)$ the family of all maximal sets of antisymmetry. Let
$f\in C(Q)$. The Shilov--Bishop decomposition theorem says that
\begin{eqnarray}\label{shibi}
f\in A\quad\Longleftrightarrow\quad f|_F\in A|_F\quad\mbox{for
all}\quad F\in\frA(A).
\end{eqnarray}
Moreover, $A|_F$ is closed in $C(F)$ for any $F\in\frA(A)$ (see
\cite[Ch. 2]{Gam} for details).

Let $M(Q)=C(Q)'$ be the space of finite regular Borel measures on
$Q$. We say that $\nu\in M(Q)$ is  {\it a probability measure} if
$\nu\geq0$ and $\nu(Q)=1$. If $\vf\in\cM_A$, then $\|\vf\|=1$ and
$\vf(\one)=1$. Hence, each norm preserving extension of $\vf$ onto
$C(Q)$ is a probability measure. Such a measure is called {\it
representing for $\vf$}. Let $\cM_\vf$ be the set of representing
measures for $\vf$. We say that a measure $\mu$ is {\it
multiplicative on $A$} if $\mu\in\cM_\vf$ for some $\vf\in\cM_A$
(in particular, a multiplicative measure has total mass 1). The
following fact is well-known.
\begin{lemma}\label{repan}
The support of a multiplicative measure is a set of antisymmetry.
\end{lemma}
\begin{proof}
Let $\vf\in\cM_A$, $\eta\in\cM_\vf$. If $f$ is real and
nonconstant on $\supp\eta$, then $\int(f-c\one)^2\,d\eta>0$ for
any $c\in\bbR$. On the other hand, if $c=\vf(f)$, then
$\int(f-c\one)^2\,d\eta=\vf((f-c\one)^2)=\vf(f-c\one)^2=0$.
\end{proof}

Let $G$ be a compact Lie group. We denote by $\tau$ several
different morphisms: the homomorphism of $G$ into the group
$\Diff(M)$ of all diffeomorphisms of the homogeneous space
\begin{eqnarray*}
M=G/H,
\end{eqnarray*}
the quasiregular representation of $G$ in $C(M)$, and its
extension onto the convolution measure algebra $M(G)$:
\begin{eqnarray*}
&\tau(g):\,x\to gx,\\
&\tau_g f(x)=f(g^{-1}x),\\
&\tau_\mu f=\int_G \tau_gf\,d\mu(g),
\end{eqnarray*}
where $g\in G$, $\mu\in M(G)$. Lie algebras are denoted by
corresponding lowercase Gothic letters. The tangent
representations of the Lie algebra $\frg$ will also be denoted by
$\tau$. We do not assume that $G$ acts on $M$ effectively. Let
$o\in M$ be the base point $eH$ of $G/H$, where $e$ denotes the
identity. The following proposition shows that any invariant
function algebra is completely determined by its ``antisymmetric
part'' $\td A$ (the restriction of $A$ onto a maximal set of
antisymmetry).
\begin{proposition}\label{shilb}
Let  $M=G/H$ be a homogeneous space of a compact Lie group $G$,
$A$ be an invariant function algebra on $M$, and $\frA(A)$ be the
family of all maximal sets of antisymmetry for $A$ in $M$. Then
there exists a closed subgroup $\td H\subseteq G$ such that
$H\subseteq \td H$ and
\begin{itemize}
\item[\rm(\romannumeral1)] $\frA(A)$ coincides with the family of
fibres of the projection $\pi:\,M\to\td M$, where $\td M=G/\td H$;
\item[\rm(\romannumeral2)] the restriction $\td A=A|_F$ of $A$ to
the fibre $F=\td H/H$ is an antisymmetric $\td H$-invariant
function algebra on $F$; \item[\rm(\romannumeral3)] $A=\{f\in
C(M):\,\tau_gf|_F\in\td A\quad\mbox{for all}\quad g\in G\}$.
\end{itemize}
In particular, $\td A$ is closed in $C(F)$ and $A$ contains each
continuous function on $M$ that is constant on every fibre; the
mapping $f\to f\circ\pi$ defines the embedding $C(\td M)\to A$.
\end{proposition}
\begin{proof}
Clearly, the family $\frA(A)$ is $G$-invariant. If $P\in\frA(A)$,
$o\in P$, $g\in G$ and $P\cap gP\neq\emptyset$, then $P\cup gP$ is
a set of antisymmetry; hence, $P=gP=g^{-1}P$. Therefore, $P=\td
Ho$ for a subgroup $\td H\subseteq G$ such that $H\subseteq\td H$.
The set $\td H$ is closed since $P$ is closed. The equality in
(\romannumeral3) is merely the Shilov--Bishop decomposition
theorem (\ref{shibi}) for invariant algebras. Further, $\td A$ is
closed in $C(F)$ by Theorem~13.1, Theorem~12.7, and Lemma~12.3 of
\cite{Gam}. The remaining assertions are obvious.
\end{proof}
By Proposition~\ref{shilb}, real functions in an invariant
function algebra separate maximal sets of antisymmetry; in
general, this is not true.

Let $\wh G$ be the family of all classes of equivalent irreducible
representations of $G$. For an invariant closed linear subspace
$L\subseteq C(M)$, let
\begin{eqnarray*}
\Sp L\subseteq \wh G
\end{eqnarray*}
be the set of the irreducible components (without multiplicities)
of $L$. If $G$ is abelian, then $\wh G=\Hom(G,\bbT)$, where $\bbT$
is the unit circle in $\bbC$ (then $M$ is also an abelian group).
Thus, $\wh G$ may be treated as a subgroup of the multiplicative
group of the Banach algebra $C(M)$. If $G=T\cong\bbT^n$ is a
torus, then there is another natural realization of $\wh
T\cong\bbZ^n$: the mapping $\chi\to -id_e\chi$, where $\chi\in\wh
T$ and $e$ is the identity of $T$, defines an embedding of $\wh T$
into the vector group $\frt'$ as a lattice, which is dual to the
kernel of $\exp$. A closed invariant subspace $A\subseteq C(G)$ is
an algebra if and only if $S=\Sp A$ is a semigroup; $A$ is
antisymmetric if and only if
\begin{eqnarray}\label{antab}
S\cap\ov S=\{\one\},
\end{eqnarray}
where the bar denotes the complex conjugation and $\one$ is the
identity of $\wh G$. The proof is evident.

There is another criterion of antisymmetry which holds for all
bi-invariant function algebras on compact groups.
\begin{theorem}\label{stara}
Let $A$ be a function algebra on a compact group $G$ which is
invariant with respect to left and right translations. The
following conditions are equivalent:
\begin{itemize}
\item[\rm(a)] $A$ is antisymmetric; \item[\rm(b)] the Haar measure
of $G$ is multiplicative on $A$.
\end{itemize}
\end{theorem}
For a proof, see \cite{Gi79} or \cite{GiP}; also,
Theorem~\ref{stara} follows from Theorem~\ref{sphan} below whose
proof uses Theorem~\ref{stara} only for tori
(Corollary\ref{mults}).

Let $N=N_G(H)$ be the normalizer of $H$ in $G$. The centralizer
$\td Z$ of $\tau(G)$ in $\Diff(M)$ may be identified with the
group $N/H$ acting in $M$ by right translations $x\to xg$ in $G$:
\begin{eqnarray}\label{deftz}
N/H\cong\td Z=Z_{\Diff(M)}(\tau(G))\subset\Diff(M).
\end{eqnarray}
Note that $\td Z$ and its identity component $\td Z^0$ are compact
Lie groups acting on $M$. Thus, the action $\tau$ can be extended
onto the group $\td G=G\times\td Z^0$. Clearly, $\td G$ preserves
the isotypical components of $\tau$.  For the multiplicity free
spaces, the Schur lemma implies the following stronger
(well-known) assertion.
\begin{lemma}\label{zinva}
If $M$ is commutative, then any $G$-invariant closed subspace of
$C(M)$ is $\td Z$-invariant and $\td G$-invariant. The group $\td
Z$ is abelian. \qed
\end{lemma}
Thus, if $M$ is commutative, then we may extend the transitive
group so that
\begin{eqnarray}\label{zintg}
\tau(Z^0)=\td Z^0, 
\end{eqnarray}
where $Z^0$ is the identity component of the centre $Z$ of $G$.
Furthermore, (\ref{zintg}) is true if the following condition
holds:
\begin{itemize}
\item[(F)] the group $\td Z$ is finite.
\end{itemize}
In the following theorem,  $M$ is not assumed to be commutative
(\cite{La99}, \cite{GL}).
\begin{theorem}\label{byvsh}
$M$ satisfies {\rm(F)} if and only if every invariant function
algebra on $M$ is self-conjugate with respect to the complex
conjugation.\qed
\end{theorem}

For a $G$-invariant function linear space $L$ on $M$, let $L_\fin$
be the set of all functions in $L$ whose translates linearly
generate a finite dimensional invariant subspace of $L$.

Let $V$ be a finite dimensional complex linear space, $\cP(V)$ be
the algebra of all (holomorphic) polynomials on $V$. For a compact
$K\subset V$, let $P(K)$ be the closure of $\cP(V)|_K$ in $C(K)$.
If $A=P(K)$, then $\cM_A$ may be identified with the {\it
polynomially convex hull} $\wh K$ of $K$, which is defined as
\begin{eqnarray}\label{defph}
\wh K=\{z\in V:\,|p(z)|\leq\sup_{\ze\in K}|p(\ze)|\quad\mbox{for
all}\quad p\in\cP(V)\}.
\end{eqnarray}
If $\wh K=K$, then $K$ is called {\it polynomially convex}. The
hull can also be defined for function algebras. Let $A$ be a
function algebra on $Q$, $K$ be a closed subset of $Q$. Then
\begin{eqnarray}\label{defah}
\wh K=\{\vf\in\cM_A:\,|\vf(f)|\leq\sup_{q\in K}|f(q)|\quad\mbox{for
all}\quad f\in A\}
\end{eqnarray}
is the {\it $A$-hull} of $K$; if $\wh K=K$, then $K$ is
$A$-convex.

If $L$ is a complex homogeneous space of $G^\bbC$, then $M$ is its
{\it real form} if it is a $G$-orbit which coincides with the set
of all fixed points of some antiholomorphic involution that
commute with $G$.

\section{A criterion of antisymmetry}
In this section, we assume {\rm(\ref{zintg})} if the contrary is
not stated explicitly.  Let $\si$ be the Haar measure of $Z^0$.
\begin{lemma}\label{pacom}
Suppose $A$ antisymmetric. Let $T\subseteq G$ be a closed subgroup
and $\nu$ be its Haar measure. If $T\supseteq Z^0$, then $\tau_\nu
A=\bbC$.
\end{lemma}
\begin{proof}
It is sufficient to prove the lemma assuming $T=Z^0$. Indeed,
suppose that  $\tau_\si A=\bbC$. Then, since $\tau_\nu$ preserves
the constant functions, the evident equality $\tau_\nu
\tau_\si=\tau_\nu$ implies $\tau_\nu A=\bbC$. Set
\begin{eqnarray*}
\td H=HZ^0,\qquad\td M=G/\td H.
\end{eqnarray*}
Since $G$ is compact, $\td M$ satisfies (F) due to (\ref{zintg}).
The space $\tau_\si A$ consists of all $Z^0$-invariant functions
in $A$. Therefore, it can be treated as an invariant subalgebra of
$C(\td M)$. By Theorem~\ref{byvsh}, $\tau_\si A$ is
self-conjugate. On the other hand, $\tau_\si A$ is antisymmetric
since $\tau_\si A\subseteq A$. Thus, $\tau_\si A=\bbC$.
\end{proof}
Let $T$ be a torus in $G$ and $\Sp_TA$ be the set of {\it weights}
of $\tau$ for $T$ in $A$:
\begin{eqnarray*}
&\Sp_TA=\{\chi\in\wh T:\,\tau_tf=\chi(t)f~~\mbox{for some}~f\in
A\setminus\{0\}~\mbox{and all}~t\in T\}.
\end{eqnarray*}
Then $f$ is said to be {\it a weight function for $\chi$}. Let
$\eta=\chi\,d\nu$, where $\nu$ is the Haar measure of $T$ and
$\chi$ is a weight. Then $\tau_\eta f$ is a weight function or
zero for every weight $\chi$ and $f\in C(M)$. Since $A_\fin$ is a
dense subalgebra of $A$ and $\tau_\eta f\in A_\fin$ for $f\in
A_\fin$, we may assume that $f\in A_\fin$.
\begin{lemma}
If $M$ is connected or $T=Z^0$, then $\Sp_T A$ is a semigroup.
\end{lemma}
\begin{proof}
Let $M_0$ be a component of $M$. Suppose that $f_1,f_2\in A_\fin$
and $f_1,f_2\neq0$ on $M_0$. Then $f_1f_2\neq0$ on $M_0$ since the
functions are real-analytic. If $f_1,f_2$ are weight functions for
$\chi_1,\chi_2$, respectively, then $f_1f_2$ is a weight function
for $\chi_1\chi_2$. Therefore, the family of such weights is a
semigroup. This proves the lemma for connected $M$. If $T=Z^0$,
then, for any weight $\chi$, the space of all weight functions for
$\chi$ is $G$-invariant. Hence, the semigroup above is independent
of the choice of $M_0$ and coincides with $\Sp_TA$.
\end{proof}
In general, $\Sp A$ is a finite union of semigroups. Let $B_T$ be
the closed linear span of $\Sp_TA$ in $C(T)$.
\begin{corollary}
If $M$ is connected or $T=Z^0$, then $B_T$ is an invariant
function algebra on $T$.\qed
\end{corollary}
\begin{lemma}\label{ancon}
If $A$ is antisymmetric, then $B_{Z^0}$ is antisymmetric.
\end{lemma}
\begin{proof}
It is sufficient to prove  (\ref{antab}) for $\Sp_{Z^0}A$. If
(\ref{antab}) is false, then there exist
$\chi\in\Sp_{Z^0}A\setminus\{\one\}$  and $u,v\in A\setminus\{0\}$
such that for all $z\in {Z^0}$ and $x\in M$
\begin{eqnarray}\label{uvchi}
u(zx)=\chi(z)u(x),\quad  v(zx)=\ov{\chi(z)}v(x).
\end{eqnarray}
Clearly, (\ref{uvchi}) is true for all $G$-translates of $u,v$.
Therefore, there exist nontrivial $G$-irreducible subspaces
$U,V\subseteq A$ such that (\ref{uvchi}) holds for all $u\in U$,
$v\in V$. The product $uv$ is $Z^0$-invariant; by
Lemma~\ref{pacom}, $uv\in\bbC$. This implies $\dim (UV)\leq1$.
Since $U$ and $V$ are $G$-invariant and nontrivial, $UV\neq0$
(hence, $\dim(UV)=1$). Let $u\in U$ and $v\in V$ be such that
$uv\neq0$. Since $uv\in\bbC$, this implies that $u(x),v(x)\neq0$
for all $x\in M$. It follows that $\dim U=\dim Uv=1$ and $\dim
V=\dim uV=1$. The corresponding one dimensional characters
$\chi_1$ and $\chi_2$ extend $\chi$ and $\ov\chi$, respectively,
onto $G$. Since $uv\in\bbC$, $\chi_1\chi_2=\one$; hence,
$\chi_1=\ov\chi_2$. Therefore, for all $g\in G$ and $x\in M$
\begin{eqnarray}\label{uvchi1}
u(gx)=\chi_1(g)u(x),\quad  v(gx)=\ov{\chi_1(g)}v(x).
\end{eqnarray}
Assuming in addition that $u(o) = v(o) = 1$, we see from
(\ref{uvchi1}) that $u = \ov{v}$. Since $A$ is antisymmetric, we
get $u=v=\one$. Then $\chi=\one$ contradictory to the assumption.
\end{proof}
\begin{corollary}\label{mults}
The operator $\tau_\si:\, A\to\bbC$ is a homomorphism.
\end{corollary}
\begin{proof}
By Lemma~\ref{ancon} and Theorem~\ref{stara}, $\si$ is
multiplicative on the algebra $B_{Z^0}$. Clearly, $B_{Z^0}$
contains all matrix elements $f_x(z)=\tau_zf(x)$, where $f\in A$,
$x\in M$, and $z\in {Z^0}$. Integrating on $z$, we get
$\tau_\si(f_1f_2)=\tau_\si(f_1)\tau_\si(f_2)$ for all $f_1,f_2\in
A$.
\end{proof}

We do not assume (\ref{zintg}) in the statements of the following
proposition and theorem.
\begin{proposition}\label{muime}
Let $A$ be an antisymmetric invariant function algebra on a
commutative homogeneous space $M$ and $\td Z^0$ be the identity
component of the group $\td Z$ defined by (\ref{deftz}). Then each
$\td Z^0$-invariant probability measure $\eta$ on $M$ is
multiplicative on $A$.
\end{proposition}
\begin{proof}
By Lemma~\ref{zinva}, we may assume (\ref{zintg}). Since $\eta$ is
$Z^0$-invariant, we have $\int f\,d\eta=\int \tau_zf\,d\eta$ for
all $z\in Z^0$, $f\in A$. By Lemma~\ref{pacom}, $\int
f\,d\eta=\tau_\si f$. Thus, the proposition follows from
Corollary~\ref{mults}.
\end{proof}

\begin{theorem}\label{sphan}
Let $G$ be a compact Lie group, $M=G/H$ be a commutative
homogeneous space, $A$ be an invariant function algebra on $M$,
and $\mu$ be the invariant probability measure on $M$. The
following assertions are equivalent:
\begin{itemize}
\item[\rm (1)] $A$ is antisymmetric; \item[\rm (2)] $\mu$ is
multiplicative on $A$; \item[\rm (3)] there is a $G$-fixed point
in $\cM_A$.
\end{itemize}
Moreover, the $G$-fixed point is unique.
\end{theorem}
\begin{proof}
By Lemma~\ref{repan} and Proposition~\ref{muime}, (1) and (2) are
equivalent. The implication (2)~$\Rightarrow$~(3) is evident.
Conversely, if $\eps\in\cM_A$ is $G$-fixed, then $\mu\in\cM_\eps$
since the set $\cM_\eps$ is weakly compact, convex, and
$G$-invariant. This proves the implication (3)~$\Rightarrow$~(2)
and the uniqueness of the fixed point.
\end{proof}
In the theorem, only the implication (1)~$\Rightarrow$~(2)
requires the commutativity of $M$. Without this assumption, the
implication is false. For example, it is not true for adjoint
orbits $M=\Ad(G)v$ and algebras $A=P(M)$ (the closure of
polynomials in $C(M)$), where $G=\SU(2)$, $v=i{\mathsf
h}+r{\mathsf e}$, $r>0$, and ${\mathsf h},{\mathsf e},{\mathsf f}$
is the standard $\sll_2$-triple (see \cite{La00}).

\section{Maximal ideal spaces and hulls of orbits}
Let us use the notation of Proposition~\ref{shilb}. Each fibre of
the fibration is a homogeneous space of a subgroup of $G$, which
is conjugate to $\td H$. For $y\in\td M$, let $\nu_y$ be the
unique invariant probability measure on the fibre $\pi^{-1}(y)$.
Set
\begin{eqnarray}\label{defho}
\al f(y)=\int f\,d\nu_y.
\end{eqnarray}
The operator $\al:\,C(M)\to C(\td M)$ is the left inverse to the
embedding $\iota f=f\circ\pi$, $\iota:\,C(\td M)\to C(M)$, where
$\pi$ is the projection $M=G/H\to G/\td H=\td M$.
\begin{theorem}\label{geifa}
Let $M=G/H$ be a commutative homogeneous space of a compact Lie
group $G$ and $A$ be an invariant function algebra on $M$. Then
\begin{itemize}
\item[\rm(a)] $\al$ is a surjective homomorphism $A\to C(\td M)$;
\item[\rm(b)] $\ker\al$ is the maximal proper $G$-invariant ideal
of $A$;  \item[\rm(c)] the dual operator $\al':\,A'\to C(M)$,
restricted onto $\cM_A$, defines a $G$-invariant fibration
$\cM_A\to\td M$, whose fibre $\wh F$ is the $A$-hull of $F=\td
H/H$.
\end{itemize}
In particular, the fibration of $\cM_A$ extends the fibration of
$M$ over $\td M$.
\end{theorem}
\begin{proof}
The homogeneous space $F=\td H/H$ is commutative since $(\td H,H)$
is a Gelfand pair if $(G,H)$ is a Gelfand pair. By definition of
$F$, the algebra $\td A=A|_F$ is antisymmetric on $F$. It follows
from Theorem~\ref{sphan} that the measure $\nu_y$ is
multiplicative on $A$ for each $y\in\td M$. Hence $\al$ is a
homomorphism. By the Shilov--Bishop theorem,
\begin{eqnarray}\label{dushb}
\iota C(\td M)\subseteq A,
\end{eqnarray}
i.e., $\al$ is surjective. Thus, (a) is true.

Clearly, $\al$ commutes with $\tau(G)$. Hence $\ker\al$ is an
invariant ideal in $A$. Since $\one\notin \ker\al$, it is proper.
Let $J$ be an invariant ideal of $A$.  If $\al J\neq0$, then $\al
J=C(\td M)$ since $\al J$ is an invariant ideal in $C(\td M)$. Let
$f\in J$ satisfy $\al f=\one$. Averaging $f$ over $\tau(\td H)$,
we get a function $u\in J$ such that $u|_F=\one$; hence, $\re
u\geq0$ near $F$. Multiplying by a suitable positive $\td
H$-invariant by right function on $M$ and averaging over
$\tau(G)$, we get a constant nonzero function in $J$. Hence, $J$
is not proper. This proves (b).

Since the projection $\al':\,\cM_A\to\td M$ is evidently
equivariant, it is sufficient to show that $\wh F$ is the pullback
of the base point $\td o=\pi({o})=\pi(F)$ of $\td M$ to prove (c).
If $\vf\in\wh F$, then $\vf$ has a representing measure $\eta$
which is concentrated on $F$ (it follows from (\ref{defah}) that
$\vf$ admits a continuous norm-preserving extension onto $C(F)$).
Hence $\al'\vf(f)=\int \al f\,d\eta=\int f\,d\nu_{\td o}$ for all
$f\in A$. Thus, $\al'\vf=\td o$. By Lemma~\ref{repan}, the support
of each representing measure for any $\vf\in\cM_A$ is contained in
some maximal set of antisymmetry. If $\al'\vf=\td o$ then this set
is $F$ (otherwise, we get a contradiction with (\ref{dushb})).
Then $\vf\in\wh F$.

The remaining assertion is clear.
\end{proof}
\begin{corollary}\label{iisel}
An invariant function algebra on a compact commutative homogeneous
space is self-conjugate with respect to the complex conjugation if
and only if it contains no proper invariant ideal.
\end{corollary}
\begin{proof}
We keep the notation of the theorem. If $\ker\al=0$, then $A$ is
self-conjugate since $A\cong C(\td M)$ by (a). If $\ker\al\neq0$,
then $A$ contains a function $f$ which is nonconstant on the
maximal set of antisymmetry $F$; then $\bar f\notin A$.
\end{proof}

Let $V$ be a finite dimensional complex vector space, $G$ be a
compact connected subgroup of $\GL(V)$, $v\in V$, $M=Gv$, and $\wh
M$ be the polynomial hull of $M$ (recall that $\wh M$ is the
maximal ideal space of the algebra $P(M)$, which is the closure of
$\cP(V)|_M$ in $C(M)$). For $v\in V$, set
\begin{eqnarray*}
\cR(v)=\left\{w\in V:\,w=\lim_{t\to+\infty}e^{it\xi}v\quad
\mbox{for some}\quad\xi\in\frg\right\}.
\end{eqnarray*}
We say that $w\in \cR(v)$ is {\it reachable from $v$}. Let
$\cP(V)^G$ be the subalgebra of $G$-invariant polynomials in
$\cP(V)$, $\cP(V)_0$ and $\cP(V)^G_0$ be the ideals  in these
algebras consisting of polynomials $p$ such that $p(0)=0$,
\begin{eqnarray}\label{defni}
\cN=\left\{v\in V:\, p(v)=0\quad\mbox{for all}\quad p\in\cP(V)^G_0\right\}.
\end{eqnarray}
The set $\cN$ is called {\it the nilpotent cone}. Since one can
get all $G$-invariant polynomials by averaging over $G$,
\begin{eqnarray}\label{defna}
&\cN=\big\{v\in V:\, \int p(gv)\,d\nu(g)=0\quad\mbox{for all}\quad
p\in\cP(V)_0\big\},
\end{eqnarray}
where $\nu$ is the Haar measure of $G$. According to the
Hilbert--Mumford criterion,
\begin{eqnarray}\label{himum}
0\in\cR(v)\quad\Longleftrightarrow\quad v\in\cN,
\end{eqnarray}
In the standard statement for the field $\bbC$, the assertion
concerns an algebraic reductive group. Any such a group is the
complexification of a compact group $G$; if $T$ is a maximal torus
in $G$, then $G^\bbC$ has the Cartan decomposition $G^\bbC=GT^\bbC
G$. This makes it possible to reduce the problem to the compact
case. If $e^{t\xi}v$, $t\in\bbR$, is periodic, then $e^{it\xi}v$
defines a holomorphic mapping $\la:\,\bbD\to V$ of the unit disc
$\bbD$ in $\bbC$ such that $\la(0)=\td v$ and $\la(\bbT)\subseteq
M$, where $\bbT=\partial\bbD$ is the unit circle.  Then
$\la(\bbD)\subseteq\wh M$. If $e^{t\xi}v$ is not periodic, then it
defines an analogous mapping of the upper halfplane and a dense
winding of a torus in $M$.

Up to the end of this section, we use the notation above.
\begin{theorem}\label{ckanm}
Suppose that $V$ contains no nonzero $G$-fixed points and $M=Gv$ is
a commutative homogeneous space. Then
\begin{itemize}
\item[\rm(a)] $P(M)=C(M)$ if and only if the complex orbit
$\cM=G^\bbC v$ is closed in $V$; \item[\rm(b)] $P(M)$ is
antisymmetric if and only if $v\in\cN$; then $\wh M\subseteq\cN$.
\end{itemize}
\end{theorem}
Both (a) and (b) are false for non-commutative spaces. However,
$M$ is not assumed to be commutative in the following lemma. Note
that each invariant function algebra $A$ evidently contains the
unique maximal invariant ideal $I$ such that $I\neq A$.
\begin{lemma}\label{iitri}
Let the complex orbit $\cM=G^\bbC v$ be closed. Then $P(M)$ contains no
proper invariant ideal.
\end{lemma}
\begin{proof}
Since the ball of radius 1 centered at $\one$ contains only
invertible elements in any Banach algebra, the maximal invariant
ideal $I\neq P(M)$ is closed in $P(M)$. Hence, $I_\fin$ is dense
in $I$. Let $\cH(\cM)$ be the algebra of all holomorphic functions
on $\cM$ with the topology of the locally uniform convergence and
$\cA(\cM)$ be the closure of $\cP(V)|_{\cM}$ in $\cH(\cM)$. It
follows that $I_\fin$ is an ideal in $\cA(\cM)_\fin$. Clearly,
each function in $\cH(\cM)$ is determined by its restriction onto
any $G$-orbit in $\cM$; in particular, this is true for functions
in $\cA(\cM)_\fin$. By \cite[Theorem~1]{GL}, $\cM$ contains a
polynomially convex $G$-orbit $N$; by \cite[Theorem~2]{GL},
$P(N)=C(N)$. Hence, $I_\fin|_N$ is a $G$-invariant ideal in
$C(N)_\fin$. This implies $I_\fin=I=0$.
\end{proof}

\begin{proof}[Proof of Theorem~\ref{ckanm}]
If $P(M)=C(M)$, then $G^\bbC v$ is closed by \cite[Theorem~2]{GL}.
Let $G^\bbC v$ be closed. Then the maximal invariant ideal of
$P(M)$ is trivial  by Lemma~\ref{iitri}. Since $P(M)$ separates
points of $M$, the Stone--Weierstrass theorem and
Corollary~\ref{iisel} imply $P(M)=C(M)$. Thus, (a) is true.

Clearly, the integrals in (\ref{defna}) are equal to $\int
p\,d\mu$, where $\mu$ is the invariant probability measure on
$M=Gv$. Therefore, $v\in\cN$ if and only if $\mu$ is
multiplicative on $\cP(V)|_M$, hence on $P(M)$. By
Theorem~\ref{sphan}, $P(M)$ is antisymmetric. It follows from
(\ref{defni}) that $\wh Q\subseteq\cN$ for any $Q\subseteq\cN$;
since $\cN$ is $G$-invariant, this implies $\wh M\subseteq\cN$.
This proves (b).
\end{proof}
\begin{proof}[Proof of Theorem~\ref{comac}]
The manifold $\cM$ can be realized as a closed $G^\bbC$-orbit for
a linear action of $G^\bbC$ in a finite dimensional complex linear
space $V$ (Matsusima \cite{Mat}, Onishchik \cite{On}). By
Theorem~\ref{ckanm}, (a), $P(M)=C(M)$; by \cite[Theorem~2]{GL},
$M$ is a real form of $\cM$.
\end{proof}
\begin{proof}[Proof of Corollary~\ref{nooth}]
According to \cite[Corollary~6]{GL}, condition (F) implies that
the real form $M$ is unique.
\end{proof}

\begin{theorem}\label{hulls}
Suppose that $M=Gv$ is a commutative homogeneous space. Then
\begin{itemize}
\item[\rm(1)] $\wh M$ contains the unique polynomially convex
$G$-orbit $\td M$; \item[\rm(2)] there exists the unique $\td
v\in\td M$ that is reachable from $v$; \item[\rm(3)] $\wh
F=\{w\in\wh M:\,\td v\in\cR(w)\}$, where $F=\td Hv$ and $\td H$ is
the stable subgroup of $\td v$ in $G$.
\end{itemize}
Moreover, $\wh M$ is equivariantly fibred over $\td M$ with the
fibre $\wh F$.
\end{theorem}
 \begin{proof}
Let $F$ and $\td H$ be defined as in Proposition~\ref{shilb};
thus, $F$ is a maximal set of antisymmetry and $\td
A=P(F)=P(M)|_F$ is an antisymmetric invariant function algebra on
$F$. Clearly, $(\td H,H)$ is a Gelfand pair if $(G,H)$ is a
Gelfand pair; hence, $F$ is commutative. Thus, we may apply
Theorem~\ref{sphan}. It follows that there exists the unique $\td
H$-fixed point $\td v\in\wh F=\cM_{\td A}$. Moreover, this implies
that $\td H$ is the stable subgroup of $\td v$ in $G$.

Set $\td M=G\td v$. Let $\nu$ be the invariant probability measure
on $F$. For each $p\in\cP(V)$, $p(\td v)=\int p\,d\nu$. Hence, the
homomorphism $\al$ of Theorem~\ref{geifa} maps $P(M)$ onto $P(\td
M)$; this implies $P(\td M)=C(\td M)$. By Theorem~\ref{ckanm},
$\td M$ is polynomially convex.  According to Theorem~\ref{geifa},
$\wh M$ is fibred over $\td M$ with the fibre $\wh F$.


Clearly, the $\td H$-equivariant linear projection in $V$ onto the
subspace of $\td H$-fixed points maps each $\td H$-orbit onto a
single point. It follows that $\wh F$ is contained in a fibre of
this projection. This fibre is an affine subspace; we may assume
that it is linear and that $\td v$ is the zero. By
Theorem~\ref{ckanm}, (b), and (\ref{himum}), $\td v\in\cR(w)$ for
all $w\in\wh F$. Hence, only the orbit $G\td v$ can be
polynomially convex; this proves (1) and (2). The equality in (3)
follows from (\ref{himum}).
\end{proof}

 \vbox{\vskip1cm
V.M. Gichev\\
gichev@ofim.oscsbras.ru\\
Omsk Branch of Sobolev Institute of Mathematics\\
Pevtsova, 13, 644099\\
Omsk, Russia}
\end{document}